\documentclass[leqno,12pt,draft]{amsart}

\usepackage{amssymb}
\usepackage{amsmath}
\usepackage{enumerate}
\usepackage{amsfonts}
\usepackage{color}

\numberwithin{equation}{section}

\newtheorem{lema}[equation]{Lemma}
\newtheorem{teo}[equation]{Theorem}

\newtheorem{coro}[equation]{Corollary}
\begin{document}
\title[Hardy spaces for Schr\"odinger operators]{A characterization of Hardy spaces
associated with certain Schr\"odinger operators}
\author{Jacek Dziuba\'{n}ski and  Jacek Zienkiewicz}
\address{Instytut Matematyczny, Uniwersytet Wroc\l awski, pl. Grunwaldzki 2/4, 50-384 Wroc\l aw, Poland} \email{jdziuban@math.uni.wroc.pl, zenek@math.uni.wroc.pl }
\thanks{The research was  supported by Polish funds for sciences, grants:   DEC-2012/05/B/ST1/00672 and DEC-2012/05/B/ST1/00692 from Narodowe Centrum Nauki.}

\maketitle

\begin{abstract}
 Let $\{K_t\}_{t>0}$ be the semigroup of linear operators generated by a
Schr\"o\-dinger operator  $-L=\Delta - V(x)$ on $\mathbb R^d$, $d\geq 3$,  where $V(x)\geq 0$ satisfies $\Delta^{-1} V\in L^\infty$.
We say that an $L^1$-function $f$ belongs to the Hardy space $H^1_L$   if the maximal function $ \mathcal M_L f(x) = \sup_{t>0} |K_tf(x)|$ belongs to $L^1(\mathbb R^d) $. We prove that the operator $(-\Delta)^{1\slash 2} L^{-1\slash 2}$ is an isomorphism of the space $H^1_L$ with the classical Hardy space $H^1(\mathbb R^d)$ whose inverse is $L^{1\slash 2} (-\Delta)^{-1\slash 2}$. As a corollary we obtain that the  space $H^1_L$ is characterized by the Riesz transforms $R_j=\frac{\partial }{\partial x_j}L^{-1\slash 2}$.
\end{abstract}


\section{Introduction and statement of the result}

 Let $K_t(x,y)$ be the integral kernels of the semigroup $\{ K_t\}_{t>0} $ of linear operators
 on $\mathbb R^d$, $d\geq 3$,  generated by a Schr\"odinger operator  $-L=\Delta -V(x)$,  where
$V(x)$ is a  non-negative locally integrable function
which satisfies
\begin{equation}\label{K}
\Delta^{-1} V(x)=-c_d \int_{\mathbb R^d}  \frac{1}{|x-y|^{d-2}} V(y)\, dy \in L^\infty (\mathbb R^d).
\end{equation}
Since $V(x)$ is non-negative, the Fenman-Kac formula implies that
\begin{equation}\label{UG}
 0\leq  K_t(x,y)\leq (4\pi t)^{-d\slash 2} e^{-|x-y|^2\slash 4t}=:P_t(x-y).
\end{equation}
It is known, see \cite{Se}, that for $V(x)\geq 0$ the condition (\ref{K}) is equivalent to the lower Gaussian bounds for  $K_t(x,y)$, that is, there are $c,C>0$ such that
\begin{equation}\label{bounds}
 ct^{-d\slash 2} e^{-C|x-y|^2\slash t}\leq K_t(x,y).
\end{equation}
We say that an $L^1$-function $f$ belongs to the Hardy space $H^1_L$ if the maximal function
$ \mathcal M_L f(x) =\sup_{t>0} |K_tf(x)|$ belongs to $L^1(\mathbb R^d)$. Then we set
$$ \| f\|_{H^1_L}=\| \mathcal M_Lf\|_{L^1(\mathbb R^d)}.$$

The Hardy spaces $H^1_L$ associated with Schr\"odinger operators with nonnegative potentials satisfying (\ref{K}) were studied in \cite{DZ4}. It was proved that the map
$ f(x) \mapsto w(x) f(x)$  is an isomorphism  of $H^1_L$ onto the classical Hardy space $H^1(\mathbb R^d)$, where
\begin{equation}\label{functionw}
w(x)=\lim_{t\to\infty } \int K_t(x,y)\, dy,
 \end{equation}
which in particular means that
\begin{equation}\label{eqwww}
\| fw\|_{H^1(\mathbb R^d)}\sim \| f\|_{H^1_L},
\end{equation}
see \cite[Theorem 1.1]{DZ4}.
 The function $w(x)$ is $L$-harmonic, that is, $ K_tw=w$, and satisfies $0<\delta \leq w(x)\leq 1$.

Let us remark that the classical real Hardy space $H^1(\mathbb R^d)$ can be thought as the space $H^1_{L} $ associated with the classical heat semigroup $e^{t\Delta}$, that is,  $L=-\Delta +V$ with $V\equiv 0$ in this case. Obviously, the constant functions are the only bounded harmonic functions for $\Delta$.

The present paper is a  continuation of \cite{DZ4}. Our goal is to study the mappings
$$L^{1\slash 2}(-\Delta)^{-1\slash 2} \ \ \text{ and } \ \  (-\Delta)^{1\slash 2}L^{-1\slash 2}$$ which turn out to be bounded on $L^1(\mathbb R^d)$ (see Lemma \ref{lemma1}). Our main result is  the following  theorem, which states another characterization of $H^1_L$.

\begin{teo}\label{main} Assume that  $L=-\Delta +V(x)$ is a Schr\"odinger operator on $\mathbb R^d$, $d\geq 3$, with a locally integrable non-negative potential  $V(x)$ satisfying (\ref{K}).
 Then the mapping $f\mapsto (-\Delta)^{1\slash 2}L^{-1\slash 2}f$ is an isomorphism of $H^1_L$ onto the classical Hardy space $H^1(\mathbb R^d)$, that is, there is a constant $C>0$ such that
 \begin{equation}\label{eq1.6}
  \| (-\Delta)^{1\slash 2} L^{-1\slash 2}  f\|_{H^1(\mathbb R^d)} \leq C\| f\|_{H^1_L}, \ \ \
 \end{equation}
 \begin{equation}\label{eq1.7}
  \|  L^{1\slash 2}  (-\Delta)^{-1\slash 2}f\|_{H^1_L} \leq C\| f\|_{H^1 (\mathbb R^d)}.
 \end{equation}
\end{teo}

As a corollary we immediately obtain the following Riesz transform characterization of $H^1_L$.
\begin{coro}\label{corollary} Under the assumptions of Theorem \ref{main}
 an $L^1$-function $f$ belongs to the space $H^1_L$ if and only if $R_jf=\frac{\partial}{\partial x_j} L^{-1\slash 2} f$ belong to $L^1(\mathbb R^d)$ for $j=1,2,...,d$. Moreover, there is a constant $C>0$ such that
 \begin{equation}\label{eq333}
  C^{-1} \| f\|_{H^1_L}\leq \| f\|_{L^1(\mathbb R^d)} +\sum_{j=1}^d\| R_jf\|_{L^1(\mathbb R^d)} \leq
  C \| f\|_{H^1_L}.
 \end{equation}
\end{coro}

\noindent{\bf Example 1.}
It is not hard to see that if for a function $V(x) \geq 0 $ defined on $\mathbb R^d$, $d\geq 3$, there is $\varepsilon >0$ such that $V\in L^{d\slash 2-\varepsilon}(\mathbb R^d)\cap L^{d\slash 2+\varepsilon}(\mathbb R^d)$, then
$V$ satisfies (\ref{K}).

\noindent{\bf Example 2.} Assume that (\ref{K}) holds for a function $V:\mathbb R^d\to [0,\infty)$, $d\geq 3$. Then $V(x_1,x_2):= V(x_1)$ defined on $\mathbb R^d\times \mathbb R^n$, $n\geq 1$,  fulfils (\ref{K}).

\

The reader interested in other results concerning Hardy spaces associated with semigroups of linear operators,
and in particular semigroups generated by Schr\"o\-din\-ger operators,
is referred to \cite{ADM}, \cite{BZ}, \cite{CZ}, \cite{DY}, \cite{DGMTZ}, \cite{DZ}, \cite{DZ5}, \cite{Hof}.

\section{Boundedness on $L^1$}
We define the operators:
\begin{equation}\begin{split}\nonumber
   & (-\Delta)^{-1} f(x)= \int_0^\infty P_tf(x)\, dt
   =c_d\int \frac{f(y)}{|x-y|^{d-2}}\, dy =:\int \Gamma_0(x-y)f(y)\, dy, \\
   & L^{-1}f(x)=\int_0^\infty K_tf(x)\, dt =: \int \Gamma(x,y)f(y)\, dy,\\
    & (-\Delta)^{-1\slash 2} f= c_1 \int_0^\infty P_tf\,\frac{ dt}{\sqrt{t}}
   =c'_d\int \frac{1}{|x-y|^{d-1}}f(y)\, dy=:\int\widetilde\Gamma_0(x-y)f(y)\, dy ,  \\
   & L^{-1\slash 2} f= c_1\int_0^\infty K_tf\,\frac{ dt}{\sqrt{t}}
   =:\int \widetilde \Gamma (x,y)f(y)\, dy,\\
\end{split}\end{equation}
where $ c_1=\Gamma (1\slash 2)^{-1}.$
 Clearly,
 \begin{equation}\begin{split}\label{eq22.1}
  0\leq \tilde \Gamma (x,y)\leq c_d'|x-y|^{-d+1}, \ \ \ 0< \Gamma(x,y)\leq c_d|x-y|^{-d+2}.
\end{split}\end{equation}

The perturbation formula asserts that
\begin{equation}\begin{split}\label{eq1}
 P_t(x-y)&=K_t(x,y)+\int_0^t \int P_{t-s}(x-z)V(z)K_s(z,y)\, dz\, ds\\
&=K_t(x,y)+\int_0^t \int K_{t-s}(x,z)V(z)P_s(z-y)\, dz\, ds.\\
 \end{split}\end{equation}

Multiplying the second inequality in  (\ref{eq1}) by $w(x)$ and integrating with respect to $dx$ we get
\begin{equation}\label{eq2.3}
 \int P_t(x-y)w(x)\, dx = w(y) +\int_{\mathbb R^d} \int_0^t w(z)V(z) P_s(z,y)\, ds \, dx,
\end{equation}
since $w$ is $L$-harmonic.
 The left-hand side of (\ref{eq2.3}) tends to a harmonic function, which is bounded from below by $\delta$ and above by 1, as $t$ tends to infinity. Thus there is  a constant  $0< c_w\leq 1$ such that
\begin{equation}\label{eq5}
 c_w= w(y) +\int_{\mathbb R^d} w(z)V(z) \Gamma_0 (z-y)\, dz.
\end{equation}
Similarly, integrating the first equation
in (\ref{eq1}) with respect to $x$  and taking limit as $t$ tends to infinity,  we get
\begin{equation}\label{eq51}
1= w(y)+ \int_{\mathbb R^d}V(z)\Gamma (z,y)\, dz.
\end{equation}

 For a reasonable function $f$ the following operators are well defined in the sense of distributions:
 \begin{equation}\begin{split}\nonumber
   &(-\Delta)^{1\slash 2} f=c_2\int_0^\infty  (P_tf-f)\frac{dt}{t^{3\slash 2}}, \  c_2=\Gamma (-1\slash 2)^{-1},\\
   &L^{1\slash 2} = c_2\int_0^\infty  (K_tf-f)\frac{dt}{t^{3\slash 2}}.\\
\end{split}\end{equation}

\begin{lema}\label{lemma1}
 There is a constant $C>0$ such that
   \begin{equation}\label{eq2.51}
   \| (-\Delta)^{1\slash 2}L^{-1\slash 2}f\|_{L^1}\leq C\| f\|_{L^1},
   \end{equation}
   \begin{equation}\label{eq2.52}
   \|  L^{1\slash 2}(-\Delta)^{-1\slash 2} f\|_{L^1} \leq C\| f\|_{L^1}.
   \end{equation}
\end{lema}
\begin{proof} From the perturbation formula (\ref{eq1}) we get
\begin{equation}\label{eq2.8}
  \begin{split}
   &(-\Delta)^{1\slash 2} L^{-1\slash 2} f (x) = c_2  \int_0^\infty (P_t - I) L^{-1\slash 2} f(x)\frac{dt}{t^{3\slash 2}}\\
   & = c_2 \int_0^\infty (P_t - K_t) L^{-1\slash 2} f(x)\frac{dt}{t^{3\slash 2}} +
    c_2\int_0^\infty (K_t - I) L^{-1\slash 2} f(x)\frac{dt}{t^{3\slash 2}}\\
    & = c_2\int_0^\infty \int_0^t \iint P_{t-s}(x-z)V(z)K_s(z,y) L^{-1\slash 2}f(y) \, dy\, dz\, ds \frac{dt}{t^{3\slash 2}}  +f(x). \\
  \end{split}
 \end{equation}
 Consider the integral kernel $W(x,u)$ of the operator
 $$f\mapsto \int_0^\infty \int_0^t \iint P_{t-s}(x-z)V(z)K_s(z,y) L^{-1\slash 2}f(y) \, dy\, dz\, ds \frac{dt}{t^{3\slash 2}},$$ that is,
 $$ W(x,u)=\int_0^\infty \int_0^t \iint P_{t-s}(x-z)V(z)K_s(z,y) \tilde \Gamma (y,u) \, dy\, dz\, ds \frac{dt}{t^{3\slash 2}}.$$
 Clearly $0\leq W(x,u)$. Integration of  $W(x,u)$   with respect to $dx$  leads to
 \begin{equation}\begin{split}\label{eq2.88}
\int W(x,u)\, dx & = \int_0^\infty \int_0^t \iint V(z)K_s(z,y) \tilde \Gamma (y,u) \, dy\, dz\, ds \frac{dt}{t^{3\slash 2}}\\
&=  2 \int_0^\infty  \iint V(z)K_s(z,y) \widetilde\Gamma (y,u) \, dy\, dz\,
\frac{ds}{\sqrt{s}} \\
&\leq 2c_1^{-1} \iint V(z)\widetilde\Gamma (z,y) \widetilde\Gamma (y,u)\, dy\, dz \\
&= 2c_1^{-1} \int V(z) \Gamma (z,u)  dz.\\
 \end{split}
 \end{equation}
 Using (\ref{eq22.1}) we see that
 $ \int W(x,u)\, dx \leq 2c_1^{-1} \| \Delta^{-1} V\|_{L^\infty},$
 which completes the proof of (\ref{eq2.51}).
 The proof of (\ref{eq2.52}) goes in the same way. We skip the details. 
\end{proof}

We finish this section by proving the following two lemmas, which will be used in the sequel.

\begin{lema}\label{lemma3}
 Assume that $f\in L^1(\mathbb R^d)$. Then
 \begin{equation}
  \int (-\Delta)^{1\slash 2} L^{-1\slash 2} f (x)\, dx = \int f(x)w(x)\, dx.
 \end{equation}
\end{lema}
\begin{proof} From (\ref{eq2.8}) and (\ref{eq2.88})  we conclude that
  \begin{equation}\begin{split}\nonumber
   \int (-\Delta)^{1\slash 2} L^{-1\slash 2} f (x)\, dx
  &=c_2 \int\int W(x,u)f(u)\, du dx  + \int f(x)\, dx\\
  & = 2c_2 c_1^{-1} \int V(z)\Gamma(z,u)f(u)\, dz \, du  +  \ \int f(x)\, dx\\
  &= \int (w(u)-1)f(u)\, du + \int f(x)\, dx,\\
  \end{split}\end{equation}
  where in the last equality we have used  (\ref{eq51}). 
\end{proof}

\begin{lema}\label{lemma2}
 Assume that $f\in L^1(\mathbb R^d)$. Then
 \begin{equation}
  \int (L^{1\slash 2} (-\Delta )^{-1\slash 2} f)(x)w(x)\, dx = c_w\int f(x)\, dx.
 \end{equation}
\end{lema}
\begin{proof} The proof is similar to that of Lemma \ref{lemma3}.
 Indeed, by the perturbation formula (\ref{eq1}) we have
\begin{equation}\nonumber
 \begin{split}
  & \int  (L^{1\slash 2} (-\Delta )^{-1\slash 2} f)(x)w(x)\, dx \\
  &=c_2\int\int_0^\infty (K_t-P_t)((-\Delta)^{-1\slash 2})f)(x)\frac{dt}{t^{3\slash 2}} w(x)\, dx\\
  &\ \ +
   c_2\int\int_0^\infty (P_t-I)((-\Delta)^{-1\slash 2})f)(x)\frac{dt}{t^{3\slash 2}} w(x)\, dx\\
  &= -c_2\int  \int_0^\infty \int_0^t \iint  w(x) K_{t-s}(x,z)V(z)\\
  &\hskip3cm \times P_s(z- y) ((-\Delta)^{-\frac{1}{2}} f)(y)\, dy dz\, ds \frac{dt}{t^{3\slash 2}} \, dx\\
   & \ \ + \int w(x)f(x)\, dx\\
 &= -c_2 \int_0^\infty \int_0^t \int_{\mathbb R^d} \int_{\mathbb R^d} w(z)V(z)P_s(z-y)((-\Delta)^{-1\slash 2} f)(y)\, dy dz\, ds \frac{dt}{t^{3\slash 2}}\\
 &\ \ +\int w(x)f(x)\, dx,
 \end{split}\end{equation}
 where in the last equality we have used that $w$ is $L$-harmonic. Integrating  with respect to $dt$  and then with respect to $ds$ yields
 \begin{equation}\nonumber
 \begin{split}
  & \int   (L^{1\slash 2} (-\Delta )^{-1\slash 2} f)(x)w(x)\, dx \\
 &= -\frac{2c_2}{ c_1 } \int \int w(z)V(z)\widetilde \Gamma_0(z-y)((-\Delta)^{-1\slash 2} f)(y)\, dy \, dz + \int f(x)w(x)\, dx \\
 &= \int w(z) V(z) \Gamma_0(z-u)f(u)\, du\, dz +\int f(x)w(x)\, dx \\
 &= \int c_w f(x)\, dx - \int w(y) f(y)\, dy +\int f(x) w(x)\, dx,
 \end{split}
\end{equation}
where in the last equality we have used  (\ref{eq5}).
\end{proof}

\section{Atoms and molecules}

Fix $1<q\leq \infty$. We say that a function $a$ is an $(1,q,w)$-atom if there is a ball $B\subset \mathbb R^d$ such that
 $\text{supp}\, a\subset B$, $\| a\|_{L^q(\mathbb R^d)}\leq |B|^{\frac{1}{q}-1}$,
 $\int a(x)w(x)\, dx=0$. The atomic norm $\| f\|_{H^1 {\rm at}, q, w}$ is defined by
\begin{equation}\label{eq3.08} \| f\|_{H^1_{{\rm  at}, q, w}}=\inf\Big\{ \sum_{j=1}^\infty |\lambda_j|\Big\},
\end{equation}
where the infimum is taken over all representations $f=\sum_{j=1}^\infty \lambda_j  a_j$,
where $\lambda_j\in\mathbb C$, $ a_j$ are $(1,q,w)$-atoms.

Clearly, if $w_0(x)\equiv 1$, then the $(1,q,w_0)$-atoms coincide with the classical $(1,q)$-atoms for the Hardy space $H^1(\mathbb R^d)$, which can be thought as $H^1_{-\Delta}$.

As a direct consequence of Theorem 1.1 of \cite{DZ4} (see (\ref{eqwww})) and the results about atomic decompositions of the classical real Hardy spaces (see, e.g., \cite{C}, \cite{Later}, \cite{Stein}),  we obtain that  the space $H^1_L$ admits atomic decomposition into $(1,q,w)$-atoms, that is, there is a constant $C_q>0$ such that
\begin{equation}
C_q^{-1} \| f\|_{H^1_{{\rm  at}, q, w}}\leq \| f\|_{H_L^1}\leq C_q\| f\|_{H^1_{{\rm  at}, q, w}}.
\end{equation}

Let $\varepsilon >0$, $1<q<\infty$. We say that a function $b$ is a $(1,q,\varepsilon, w)$-molecule associated with a ball $B=B(x_0, r)$ if
\begin{equation}\label{eq31}
\Big(\int_B |b(x)|^q\, dx \Big)^{\frac{1}{ q}} \leq |B|^{\frac{1}{q}-1}, \ \ \Big( \int_{2^kB\setminus 2^{k-1}B} |b(x)|^q\, dx\Big)^{\frac{1}{ q}} \leq |2^kB|^{\frac{1}{q}-1} 2^{-\varepsilon k}
\end{equation}
and
\begin{equation}\label{eq33}
 \int b(x)w(x)\, dx =0.
\end{equation}
Obviously every $(1,q,w)$-atom is a $(1,q,\varepsilon, w)$-molecule. It is also not hard to see that for fixed  $q>1$ and $\varepsilon>0$  there is a constant $C>0$ such that every $(1,q,\varepsilon, w)$ molecule $b$  can be decomposed into a  sum
$$ b(x)=\sum_{n=1}^\infty \lambda_n a_n, \ \ \sum_{n=1}^\infty |\lambda_n|\leq C, $$
where  $\lambda_n\in \mathbb C$, $a_n$ are $(1,q,w)$-atoms.

The following lemma is easy to prove.
\begin{lema}\label{lemmaLq}
 Let $1<q<\infty$, $\delta, \varepsilon >0$ be such that $\delta>d(1-\frac{1}{q})+\varepsilon$. Then there is a constant $C>0$ such that if $b(x)$ satisfies (\ref{eq33}) and
 \begin{equation}
 \Big( \int \Big| b(x)\Big(1+\frac{|x-y_0|}{r}\Big)^\delta\Big|^q\, dx \Big)^{1\slash q}\leq \frac{r^{-d+d\slash q}}{C},
 \end{equation}
 then $b$ is a $(1,q,\varepsilon, w)$-molecule associated with $B(y_0,r)$.
\end{lema}

In order to prove Theorem \ref{main} we shall use  general results about Hardy spaces associated with Schr\"odinger operators with non-negative potentials which were proved in \cite{DZ1}. Let $\{T_t\}_{t>0}$ be a semigroup of linear operators generated by a
Schr\"odinger operator $-\mathcal L=\Delta - \mathcal V(x)$ on $\mathbb R^d$, where $\mathcal V(x)$ is
a non-negative locally integrable potential. The Hardy space $H^1_{\mathcal L}$ is define by means of the maximal function, that is,
$$H^1_{\mathcal L}= \{ f\in L^1(\mathbb R^d): \| f\|_{H^1_{\mathcal L}}:=\| \sup_{t>0} |T_tf(x)|\|_{L^1(\mathbb R^d)}<\infty\}.$$
 We say that a function $\mathbf a$ is a generalized $(1,\infty,\mathcal L)$-atom for the Hardy space $H^1_{\mathcal L}$ if there is a ball $B=B(y_0,r)$ and a function $\mathbf b$ such that
 $$\text{supp}\, \mathbf b\subset B, \ \ \ \|\mathbf b\|_{L^\infty}\leq |B|^{-1},  \ \ \ \mathbf a=(I-T_{r^2}) \mathbf b.$$
  Then we say that $\mathbf a$ is associated with the ball $B(y_0,r)$.
It was proved in Section 6 of \cite{DZ1} that the space $H^1_{\mathcal L}$ admits atomic decomposition with the generalized $(1,\infty,\mathcal L)$-atoms, that is, $\| f\|_{H^1_{\mathcal L}} \sim \| f\|_{H^1_{\mathbf{ at},\infty, \mathcal L }}$, where  the norm $\| f\|_{H^1_{\mathbf{ at},\infty, \mathcal L }}$ is defined as in (\ref{eq3.08}) with $a_j(x)$ replaced by the general $(1,\infty,\mathcal L)$-atoms $\mathbf a_j(x)$.

\begin{lema}\label{lemma3.7} There is a constant $C>0$ such that for every $\mathbf a$ being a generalized $(1,\infty,\mathcal L)$ atom associated with $B(y_0,r)$ one has
$$|\mathcal L^{-1\slash 2} \mathbf a(y)|\leq Cr^{1-d}\Big(1+\frac{|y-y_0|}{r}\Big)^{-d}.$$
\end{lema}

\begin{proof} The proof follows from functional calculi  (see, e.g., \cite{Hebisch}).  Note that  $\mathcal L^{-1\slash 2} \mathbf a=m_{(r)}(\mathcal L)\mathbf b$ with $m_{(r)}(\lambda)=r(r^2\lambda)^{-1\slash 2} (e^{-r^2\lambda} -1)$ and  $\mathbf b$ such that $\text{supp}\, \mathbf b\subset B(y_0,r)$, $\| \mathbf b  \|_{L^\infty}\leq |B(y_0,r)|^{-1}$. From \cite{Hebisch} we conclude that there is a constant $C>0$ such that for every $r>0$ one has
$$ m_{(r)}(\mathcal L) f(x)=\int_{\mathbb R^d} m_{(r)}(x,y)f(y)\, dy,$$
with $m_{(r)}(x,y)$ satisfying
\begin{equation}\label{mmm}
 |m_{(r)}(x,y)|\leq Cr^{1-d} \Big(1+\frac{|x-y|}{r}\Big)^{-d}.
\end{equation}
Now the lemma can be easily deduced from (\ref{mmm})
and the size and support property of $\mathbf b$. 
\end{proof}

\section{Proof of Theorem \ref{main}}

For  real numbers $n>2$, $\beta>0$ let
$$g(x)=(1+|x|)^{-n-\beta}, \ \ \ g_s(x)=s^{-n\slash 2} g\big(\frac{x}{ \sqrt{s}}\big).$$ One can easily check that
\begin{equation}\label{eq41}
 \int_0^t g_s(x)\, ds \leq
 C |x|^{2-n}\Big(1+\frac{|x|}{\sqrt{t}}\Big)^{-2-\beta};
\end{equation}
\begin{equation}\label{eq4.2}
 \int_{r^2}^{\infty } g_s(x)\, ds  \leq
 C r^{2-n} \Big(1+\frac{|x|}{r}\Big)^{-n+2} \ \ \text{for} \ r>0.
\end{equation}
Moreover, it is easily to verify that for $1<q<\infty$, $d(1-\frac{1}{q})<\alpha\leq d, \ \beta>0$ one has
\begin{equation}\label{eq4.4}
  \Big\| |x|^{\alpha-d} \Big(1+\frac{|x|}{\sqrt{t}}\Big)^{-d-\beta}\Big\|_{L^q(\mathbb R^d, \,dx)}= C_{\alpha ,\beta} t^{(\alpha-d+d\slash q)\slash 2}
\end{equation}
and
\begin{equation}\label{eq4.5}
 \int |z-y|^{2-d} \Big(1+\frac{|z-y|}{r}\Big)^{-\beta} \Big(1+\frac{|y|}{r}\Big)^{-d+\gamma} \, dy\leq Cr^2 \Big(1+\frac{|z|}{r}\Big)^{-d+\gamma+2-\beta}
\end{equation}
for $0<\gamma <\beta<  2$ .

\begin{lema} Assume that $V(x)$ satisfies the assumptions of Theorem \ref{main}. Then for $0<\gamma\leq 2$ and $r>0$ one has
\begin{equation}\label{eq4.3}
 \int_{\mathbb R^d} V(z) \Big(1+\frac{|z-y|}{r}\Big)^{-d+\gamma} \, dz \leq  c_d^{-1} r^{d-2}\| \Delta^{-1}V\|_{L^\infty}.
\end{equation}
\end{lema}
\begin{proof} The left-hand side of (\ref{eq4.3}) is bounded by
\begin{equation}\nonumber
\begin{split}
 \int_{|z-y|\leq r} V(z)\Big(\frac{r}{|z-y|}\Big)^{d-2} \, dz & + \int_{|z-y|>r} V(z) \Big(\frac{|z-y|}{r}\Big)^{-d+2}\, dz\\
 &\leq c_d^{-1} r^{d-2}\| \Delta^{-1}V\|_{L^\infty}. \\
\end{split}
\end{equation} 
\end{proof}

\begin{proof}[Proof of Theorem \ref{main}] We already have known that the operators $(-\Delta)^{1\slash 2}L^{-1\slash 2}$ and $L^{1\slash 2}(-\Delta)^{-1\slash 2}$ are bounded on $L^1(\mathbb R^d)$.
 It suffices to prove (\ref{eq1.6}) and (\ref{eq1.7}).
 Set $\gamma=\frac{1}{10}$ and fix $q>1$ and $\varepsilon>0$ such that $\gamma>d(1-\frac{1}{q})+\varepsilon$.
 Set $w_0(x)\equiv 1$. According to the atomic and molecular decompositions (see Section 3) the proof of (\ref{eq1.6}) will be done if we verify that
$(-\Delta)^{1\slash 2} L^{-1\slash 2} \mathbf a$ is a multiple of a $(1,q,\varepsilon, w_0)$-molecule for every generalized $(1,\infty, L)$-atom $\mathbf a$ with a multiple  constant independent of $\mathbf a$. Identical arguments can be then applied to show that $L^{1\slash 2} (-\Delta)^{-1\slash 2}  \mathbf a$ is a $(1,q,\varepsilon, w)$-molecule for $\mathbf a$ being a generalized atom for the classical Hardy space $H^1(\mathbb R^d)=H^1_{-\Delta}$ with a multiple constant independent of $\mathbf a$.

Let $\mathbf a=(I-K_{r^2})\mathbf b$ be a generalized $(1,\infty, L)$-atom for $H^1_L$ associated with $B(y_0,r)$. By Lemma \ref{lemma3}, since $\int w(x)\mathbf a(x)\, dx=0$, we have that
$$\int (-\Delta)^{1\slash 2} L^{-1\slash 2} \mathbf a(x)\, dx =0.$$
 Set
\begin{equation}\begin{split}
J(x) &= \int_0^\infty \int_0^t \iint P_{t-s}(x-z)V(z)K_s(z,y) (L^{-1\slash 2}\mathbf a)(y) \, dy\, dz\, ds \frac{dt}{t^{3\slash 2}}\\
&= \int_0^{r^2 } \int_0^{t} \iint ... + \int_{r^2 }^\infty \int_0^{t\slash 2} \iint
+ \int_{r^2}^{\infty}  \int_{t\slash 2}^t \iint... \\
&=J_1(x)+J_2(x) +J_3(x).
\end{split}
\end{equation}
Thanks to (\ref{eq2.8}) and  Lemma \ref{lemmaLq}  it suffices to show that there is a constant  $C_q>0$,  independent of $\mathbf a(x)$ such that
\begin{equation}\label{eq4.7}
\Big\|\Big(1+\frac{|x-y_0|}{r}\Big)^\gamma J(x)\Big\|_{L^q(\mathbb R^d)}\leq C_qr^{-d+d\slash q}.
\end{equation}

Applying Lemma \ref{lemma3.7} and (\ref{eq41}) with $n=d+1$,  we obtain
\begin{equation}\begin{split}
|J_1(x)| &= \Big|\int_0^{r^2} \int_0^t \iint P_{t-s}(x-z)V(z)K_s(z,y) (L^{-1\slash 2}\mathbf a)(y) \, dy\, dz\, ds \frac{dt}{t^{3\slash 2}}\Big|\\
&\leq C \int_0^{r^2} \int_0^t \int P_{t-s}(x-z)V(z) r^{1-d} \Big(1+\frac{|z-y_0|}{r}\Big)^{-d} \,  dz\, ds \frac{dt}{t^{3\slash 2}}\\
&\leq C \int_0^{r^2}  \int P_{s}(x-z)V(z) r^{1-d}\Big(1+\frac{|z-y_0|}{r}\Big)^{-d} \,  dz\, \frac{ds}{\sqrt{s}} \\
&\leq C_N  \int |x-z|^{1-d}\Big(1+\frac{|x-z|}{r}\Big)^{-N}V(z) r^{1-d}\Big(1+\frac{|z-y_0|}{r}\Big)^{-d} \,  dz. \\
\end{split}
\end{equation}
Consequently,
\begin{equation}\begin{split}
&|J_1(x)| \Big(1+\frac{|x-y_0|}{r}\Big)^\gamma \\
&\leq C_N r^{1-d} \int |x-z|^{-d+1}\Big(1+\frac{|x-z|}{r}\Big)^{-N+\gamma}V(z) \Big(1+\frac{|z-y_0|}{r}\Big)^{-d+\gamma} \,  dz.\\
\end{split}
\end{equation}
  Therefore, using the Minkowski integral inequality together with (\ref{eq4.4}) and (\ref{eq4.3}), we get
  \begin{equation}\label{eq4.10}
   \Big\| J_1(x) \Big(1+\frac{|x-y_0|}{r}\Big)^\gamma\Big\|_{L^q(dx)} \leq Cr^{-d+d\slash q}.
  \end{equation}

In order to estimate $J_2(x)$ we use Lemma \ref{lemma3.7} and (\ref{eq41}) with $n=d$  to obtain

\begin{equation}\begin{split}
& |J_{2}(x)| \Big(1+\frac{|x-y_0|}{r}\Big)^\gamma \\
 & \leq C \int_{r^2}^{\infty}  \Big(1+\frac{|x-y_0|}{r}\Big)^\gamma \int_0^{t\slash 2} \iint
 t^{-d\slash 2}e^{-c|x-z|^2\slash t} V(z)\\
 &\ \ \ \times K_s(z,y) r^{1-d} \Big(1+\frac{|y-y_0|}{r}\Big)^{-d} \, dy \,  dz\, ds \frac{dt}{t^{3\slash 2}}\\
&\leq C \int_{r^2}^\infty \iint t^{(2\gamma-d-3)\slash 2}e^{-c|x-z|^2\slash t}  V(z) \\
&\ \ \ \times |z-y|^{2-d}\Big(1+\frac{|z-y|}{\sqrt{t}}\Big)^{-N+\gamma}r^{1-d-2\gamma } \Big(1+\frac{|y-y_0|}{r}\Big)^{-d+\gamma}\, dy\, dz\, dt . \\
\end{split}
\end{equation}
Setting $N=\beta+\gamma$ with $0<\gamma <\beta<  2$ and applying the Minkowski integral inequality together with  (\ref{eq4.5}) and (\ref{eq4.3}) we conclude that
\begin{equation}\begin{split}\label{eq4.12}
& \Big\| J_{2}(x) \Big(1+\frac{|x-y_0|}{r}\Big)^\gamma\Big\|_{L^q(dx)} \\
 & \leq C \int_{r^2}^{\infty}\iint  t^{-(d+3-2\gamma-d\slash q)\slash 2}  V(z)\\
 &\ \ \ \times  |z-y|^{2-d}
   \Big(1+\frac{|z-y|}{\sqrt{t}}\Big)^{-\beta }r^{1-d-2\gamma } \Big(1+\frac{|y-y_0|}{r}\Big)^{-d+\gamma}\, dy\, dz\, dt \\
   & \leq C \int_{r^2}^{\infty}\iint  t^{-(d+3-2\gamma-d\slash q)\slash 2}  V(z)\\
 &\ \ \ \times  |z-y|^{2-d}
   \Big(1+\frac{|z-y|}{r}\Big)^{-\beta }\Big(\frac{\sqrt{t}}{r}\Big)^\beta r^{1-d-2\gamma } \Big(1+\frac{|y-y_0|}{r}\Big)^{-d+\gamma}\, dy\, dz\, dt \\
 &\leq C \int r^{-2d +2+d\slash q} V(z)\Big(1+\frac{|z-y_0|}{r}\Big)^{-d+2+ \gamma -\beta}\, dz \\
 &\leq C r^{-d+d\slash q}.
\end{split}
\end{equation}
By Lemma \ref{lemma3.7} and (\ref{eq41}) with $n=d$,  we have
\begin{equation}\begin{split}\nonumber
&|J_{3}(x)|\\
&\leq C \int_{r^2}^\infty \int_{\frac{t}{ 2}}^{t}\iint P_{t-s} (x-z) V(z) t^{-\frac{d}{2}} e^{-\frac{c|z-y|^2}{t}} \\
& \hskip2.5cm \times \Big(1+\frac{|y-y_0|}{r}\Big)^{-d}r^{1-d}\, dy\, dz\, ds\, \frac{dt}{t^{\frac{3} {2}}}\\
&\leq C_N\int_{r^2}^\infty
\iint  |x-z|^{2-d} \Big(1+\frac{|x-z|}{\sqrt{t}}\Big)^{-N} V(z)\\
&\ \ \ \times
t^{-\frac{d}{2}} e^{-c|z-y|^2\slash t}
\Big(1+\frac{|y-y_0|}{r}\Big)^{-d}r^{1-d}\, dy\, dz\, ds\, \frac{dt}{t^{3\slash 2}}.
\end{split}\end{equation}
Hence,
\begin{equation}\begin{split}\nonumber
&|J_{3}(x)|\Big(1+\frac{|x-y_0|}{r}\Big)^\gamma\\
&\leq C\int_{r^2}^\infty
\iint  |x-z|^{2-d} \Big(1+\frac{|x-z|}{\sqrt{t}}\Big)^{-N+\gamma}t^\gamma  V(z)\\
&\ \ \ \times
t^{-\frac{d}{2}} e^{-c'|z-y|^2\slash t}
\Big(1+\frac{|y-y_0|}{r}\Big)^{-d+\gamma}r^{1-d-2\gamma}\, dy\, dz\, ds\, \frac{dt}{t^{3\slash 2}}.
\end{split}\end{equation}
By Minkowski's integral inequality combined with (\ref{eq4.4})  we arrive to
\begin{equation}
\begin{split}\nonumber
&\Big\|J_3(x)\Big(1+\frac{|x-y_0|}{r}\Big)^\gamma\Big\|_{L^q(dx)} \\
 &\leq \int_{r^2}^\infty \iint t^{(-d +2+d\slash q)\slash 2  +\gamma -3\slash 2} V(z) \\
 &\hskip1.5cm \times
t^{-d\slash 2} e^{-c'|z-y|^2\slash t}
\Big(1+\frac{|y-y_0|}{r}\Big)^{-d+\gamma}r^{1-d-2\gamma}\, dy\, dz\,  dt.\\
\end{split}\end{equation}
Application of  (\ref{eq4.2}) with $n=2d+1-\frac{d}{q}-2\gamma$ and then  (\ref{eq4.3}) yields
\begin{equation}
\begin{split}\nonumber
&\Big\|J_3(x)\Big(1+\frac{|x-y_0|}{r}\Big)^\gamma\Big\|_{L^q(dx)} \\
&\leq C \iint r^{2-3 d +d\slash q}V(z)\Big(1+\frac{|z-y|}{r}\Big)^{-2d+1+d\slash q +2\gamma}
\Big(1+\frac{|y-y_0|}{r}\Big)^{-d+\gamma}\, dy\, dz \\
&\leq \int r^{2-2d+d\slash q} V(z)\Big(1+\frac{|z-y_0|}{r}\Big)^{-2d+1+d\slash q +3\gamma}\, dz \\
&\leq Cr^{-d+d\slash q}.
\end{split}\end{equation}
The above inequality together with (\ref{eq4.10}) and (\ref{eq4.12}) gives desired  (\ref{eq4.7}) and, consequently, the proof of (\ref{eq1.6}) is complete.

Let us note that in the proof (\ref{eq1.6}) we use only Lemmas \ref{lemma3}, \ref{lemma3.7}, and the upper Gaussian bounds for the kernels. The proof of (\ref{eq1.7}) goes identically to that of (\ref{eq1.6}) by replacing Lemma \ref{lemma3} by Lemma \ref{lemma2}. 
\end{proof}
\section{Proof of the Riesz transform characterization of $H^1_L$}

\begin{proof}[Proof of Corollary \ref{corollary}]
 Assume that $f\in H^1_L$.  Then, thanks to Theorem \ref{main},
   there is $g\in H^1(\mathbb R^d)$ such that  $f=L^{1\slash 2} (-\Delta)^{-1\slash 2} g$. By the characterization of the classical Hardy space $H^1(\mathbb R^d)$ by the Riesz transforms  we have
  \begin{equation}\label{eq66}
  \frac{\partial }{\partial x_j } L^{-1\slash 2} f =\frac{\partial}{\partial x_j} L^{-1\slash 2} L^{1\slash 2} (-\Delta)^{-1\slash 2} g=\frac{\partial}{\partial x_j} (-\Delta)^{-1\slash 2} g\in L^1(\mathbb R^d).
  \end{equation}
  Conversely, assume that for $f\in L^1(\mathbb R^d)$ we have $\frac{\partial }{\partial x_j} L^{-1\slash 2}f\in L^1(\mathbb R^d)$ for $j=1,2,...,d$. Set $g=(-\Delta)^{1\slash 2} L^{-1\slash 2} f$. Then by Lemma \ref{lemma1}, $g\in L^1(\mathbb R^d)$ and
  \begin{equation}\label{eq666}
  \frac{\partial}{\partial x_j} (-\Delta)^{-1\slash 2} g= \frac{\partial}{\partial x_j} (-\Delta)^{-1\slash 2} (-\Delta)^{1\slash 2} L^{-1\slash 2} f= \frac{\partial}{\partial x_j} L^{-1\slash 2} f\in L^1(\mathbb R^d),
  \end{equation}
  which implies that $ g\in H^1(\mathbb R^d)$. Consequently, by Theorem \ref{main}, $f\in H^1_L$. Finally (\ref{eq333}) can be deduced from (\ref{eq66}), (\ref{eq666}), and Theorem \ref{main}.  
\end{proof}



\end{document}